\documentclass[11pt,english]{article}

\usepackage[margin = 2.5cm]{geometry}

\usepackage{amsmath}
\usepackage{amsthm}
\usepackage{amssymb}
\usepackage{mathtools}
\usepackage{graphicx}
\usepackage{xcolor}
\usepackage[hidelinks, bookmarks = false]{hyperref}
\usepackage{enumitem}
\usepackage{bbm}

\usepackage{caption}
\usepackage[square,sort,comma,numbers]{natbib}
\usepackage{setspace}
\usepackage{cleveref}
\theoremstyle{plain}

\newtheorem{theorem}{Theorem}[section]
\crefname{theorem}{Theorem}{Theorems}
\Crefname{theorem}{Theorem}{Theorems}

\newtheorem*{lemma*}{Lemma}
\newtheorem{lemma}[theorem]{Lemma}
\crefname{lemma}{Lemma}{Lemmas}
\Crefname{lemma}{Lemma}{Lemmas}

\newtheorem*{claim*}{Claim}
\newtheorem{claim}[theorem]{Claim}
\crefname{claim}{Claim}{Claims}
\Crefname{claim}{Claim}{Claims}

\newtheorem{proposition}[theorem]{Proposition}
\crefname{proposition}{Proposition}{Propositions}
\Crefname{proposition}{Proposition}{Propositions}

\crefname{corollary}{Corollary}{Corollaries}
\Crefname{corollary}{Corollary}{Corollaries}

\newtheorem{conjecture}[theorem]{Conjecture}
\crefname{conjecture}{Conjecture}{Conjectures}
\Crefname{conjecture}{Conjecture}{Conjectures}

\crefname{question}{Question}{Questions}
\Crefname{question}{Question}{Questions}

\crefname{observation}{Observation}{Observations}
\Crefname{observation}{Observation}{Observations}

\crefname{example}{Example}{Examples}
\Crefname{example}{Example}{Examples}

\theoremstyle{definition}

\crefname{problem}{Problem}{Problems}
\Crefname{problem}{Problem}{Problems}

\newtheorem{definition}[theorem]{Definition}
\crefname{definition}{Definition}{Definitions}
\Crefname{definition}{Definition}{Definitions}

\theoremstyle{remark}

\crefname{remark}{Remark}{Remarks}
\Crefname{remark}{Remark}{Remarks}

\usepackage{xpatch}
\xpatchcmd{\proof}{\itshape}{\normalfont\proofnamefont}{}{}
\newcommand{\proofnamefont}{}
\renewcommand{\proofnamefont}{\bfseries}

\usepackage{framed}

\newcommand{\remove}[1]{}

\newcommand{\al}{\alpha}
\newcommand{\be}{\beta}

\newcommand{\eps}{\varepsilon}

\newcommand{\Bin}{\text{Bin}}

\DeclareMathOperator{\ex}{ex}

\newcommand{\C}{\mathcal{C}}

\newcommand{\F}{\mathcal{F}}

\newcommand{\Pa}{\mathcal{P}}

\title{Bipartite-ness under smooth conditions}

\author{
        Tao Jiang
        \thanks{Department of Mathematics, Miami University, Oxford, OH 45056, USA. Email: \texttt{jiangt}@\texttt{miamioh.edu}.
        Research supported by National Science Foundation grant DMS-1855542.}
    \and
	    Sean Longbrake\thanks{
	    Department of Mathematics, Miami, Oxford, OH 45056, USA. Current address: Department of Mathematics, Emory University, Atlanta, GA 30322, Email: \texttt{
	    sean.longbrake}@\texttt{emory.edu}.
        Research supported by National Science Foundation grant DMS-1855542.}	
    \and
    Jie Ma \thanks{
School of Mathematical Sciences, University of Science and Technology of China, Hefei, Anhui 230026, China.
Email: \texttt{jiema}@\texttt{ustc.edu.cn}.
Research supported by the National Key R and D Program of China 2020YFA0713100,
National Natural Science Foundation of China grant 12125106, and Anhui Initiative in Quantum Information Technologies grant AHY150200.
}
}

\begin{document}

\maketitle

\begin{abstract}

	\setlength{\parskip}{\medskipamount}
    \setlength{\parindent}{0pt}
    \noindent

Given a family $\F$ of bipartite graphs, the {\it Zarankiewicz number} $z(m,n,\F)$ is the maximum number of edges in an $m$ by $n$ bipartite graph $G$
that does not contain any member of $\F$ as a subgraph (such $G$ is called {\it $\F$-free}).
For $1\leq \beta<\alpha<2$, a family $\F$ of bipartite graphs is $(\alpha,\beta)$-{\it smooth} if for some $\rho>0$ and every $m\leq n$, $z(m,n,\F)=\rho m n^{\alpha-1}+O(n^\beta)$.
Motivated by their work on a conjecture of Erd\H{o}s and Simonovits on compactness and a classic result of Andr\'asfai, Erd\H{o}s and S\'os,
in \cite{AKSV} Allen, Keevash, Sudakov and Verstra\"ete proved that for any $(\alpha,\beta)$-smooth family $\F$,
there exists $k_0$ such that for all odd $k\geq k_0$ and sufficiently large $n$, any $n$-vertex $\F\cup\{C_k\}$-free graph with minimum degree at least $\rho(\frac{2n}{5}+o(n))^{\alpha-1}$ is bipartite.

In this paper, we strengthen their result by showing that for every real $\delta>0$, there exists $k_0$ such
that for all odd $k\geq k_0$ and sufficiently large $n$, any $n$-vertex $\F\cup\{C_k\}$-free graph with minimum degree at least $\delta n^{\alpha-1}$ is bipartite.
Furthermore, our result holds under a more relaxed notion of smoothness, which include the families $\F$ consisting of the single graph $K_{s,t}$ when $t\gg s$.
We also prove an analogous result for $C_{2\ell}$-free graphs for every $\ell\geq 2$, which complements a result of Keevash, Sudakov and Verstra\"ete in \cite{KSV}.
\end{abstract}


\section{Introduction} \label{sec:intro}
Given a family $\F$ of graphs, a graph $G$ is called {\it $\F$-free} if $G$ does not contain any member of $\F$ as a subgraph.
If $\F$ consists of a single graph $F$ then we simply say that $G$ is {\it $F$-free}.
The {\it Tur\'an number} of $\F$, denoted by $\ex(n, \F)$, is the maximum possible number of edges in an $n$-vertex $\F$-free graph.
As is well known, this function is well-understood when $\F$ consists only of non-bipartite graphs due to the celebrated Erd\H{o}s-Stone-Simonovits theorem \cite{Erdos-Stone, ES-limit} but is generally open when $\F$ contains bipartite graphs. For a family of graphs $\F$, a closely related notion is the so-called
{\it Zarankiewicz number} $z(n,\F)$, which is defined to be the maximum number of edges in
an $n$-vertex $\F$-free bipartite graph. More generally, we denote by $z(m,n,\F)$ the
maximum number of edges in an $m$ by $n$ bipartite graph that is $\F$-free.
In a seminal paper \cite{ES-compactness}, Erd\H{o}s and Simonovits raised a number of intriguing conjectures on Tur\'an numbers for bipartite graphs.
One of them is the following (Conjecture 3 in \cite{ES-compactness}).
Given a positive odd integer $k$, let $\C_k$ denote the family of all odd cycles of length at most $k$. Throughout this paper, we write $f(n)\sim g(n)$ for two functions $f, g: \mathbb{N}\to \mathbb{R}$ if $\lim_{n\to \infty} f(n)/g(n)=1$.

\begin{conjecture} [Erd\H{o}s-Simonovits \cite{ES-compactness}] \label{conj:ES}
Given any finite family $\F$ of graphs, there exists an odd integer $k$ such that as $n\to \infty$
\[\ex(n,\F\cup \C_k)\sim z(n,\F).\]
\end{conjecture}

Erd\H{o}s and Simonovits \cite{ES-compactness} verified the conjecture for $\F=\{C_4\}$ by showing that
$\ex(n, \{C_4,C_5\})\sim z(n, C_4)\sim (\frac{n}{2})^{\frac{3}{2}}$.
Keevash, Sudakov and Verstra\"ete \cite{KSV} further confirmed this conjecture for $\F_\ell:=\{C_4,C_6,\dots, C_{2\ell}\}$ where $\ell\in \{2,3,5\}$ in stronger forms and proved a related result for the chromatic number of $\F_\ell\cup \{C_k\}$-free graphs of minimum degree $\Omega(n^{1/\ell})$.
In a subsequent paper \cite{AKSV}, Allen, Keevash, Sudakov and Verstra\"ete provided a general approach to Conjecture~\ref{conj:ES} (using Scott's sparse regularity lemma \cite{Scott}),
which works for the following families of bipartite graphs.

\begin{definition}\label{dfn:smooth}
Let $\alpha, \beta$ be reals with $2>\alpha>\beta\geq 1$. Let $\F$  be a family of bipartite graphs.
If there exists some $\rho> 0$ such that for every $m\leq n$,
\[z(m,n, \F)=\rho mn^{\alpha-1}+O(n^\beta)\]
holds, then we say that $\F$ is {\it $(\alpha,\beta)$-smooth} with relative density $\rho$.
We call a bipartite family $\F$ {\it smooth} if it is $(\al,\beta)$-smooth for some $\al$ and $\beta$.
\end{definition}

It is easy to see that for any $(\alpha,\beta)$-smooth family $\F$, we have $z(n,\F)=\rho(n/2)^\alpha+O(n^\beta).$

Before we mention the results of \cite{AKSV}, let us discuss some known examples of smooth families.
Improving results of  K\"ov\'ari-S\'os-Tur\'an \cite{KST},
F\"uredi~\cite{furedi-zarankewicz} showed that if $m\leq n$ and $s,t\in\mathbb{N}$ then
\begin{equation}\label{equ:z(Kst)}
z(m,n, K_{s,t}) \leq (t-s+1)^{1/s} mn^{1-1/s}+ sm +s n^{2-2/s}.
\end{equation}
This together with the constructions of Brown \cite{brown} and F\"uredi \cite{furedi-96} shows that $K_{2,t}$ and $K_{3,3}$ are smooth families
(see \cite{AKSV}).
Allen, Keevash, Sudakov and Verstra\"ete \cite{AKSV} also showed that $\{K_{2,t}, B_t\}$ is smooth, where $B_t$ consists of $t$ copies of $C_4$ sharing an edge (and no other vertices).
However, it is not known if $K_{s,t}$ is smooth for any $s\geq 3$ and $t\geq 4$ and if $C_{2\ell}$ is smooth for any $\ell\geq 3$, due to a lack of constructions that asymptotically match upper bounds on Zarankiewicz numbers.
We would like to point out that not all families of bipartite graphs are smooth -- in the concluding remarks we provide an example of bipartite graphs which are not smooth.

The main result of Allen, Keevash, Sudakov and Verstra\"ete  \cite{AKSV} is as follows.
A family $\mathcal{G}$ of graphs is {\it near-bipartite} if every graph $G\in \mathcal{G}$ has a bipartite subgraph $H$ such that $e(G)\sim e(H)$ as $|V(G)|\to \infty$.

\begin{theorem}[Allen-Keevash-Sudakov-Verstra\"ete~\cite{AKSV}] \label{thm:AKSV-smooth}
Let $\F$ be an $(\alpha,\beta)$-smooth family with $2>\alpha>\beta\geq 1$.
There exists $k_0$ such that if $k\geq k_0 \in \mathbb{N}$ is odd, then
the family of all extremal $\F\cup \{C_k\}$-free graphs is near-bipartite and, in particular, $\ex(n,\F\cup \{C_k\})\sim z(n,\F)$.
\end{theorem}

The authors \cite{AKSV} also raised a question whether the extremal $n$-vertex $\F\cup \{C_k\}$-free graph in Theorem~\ref{thm:AKSV-smooth} is exactly bipartite when $n$ is sufficiently large.
Motivated by the classic result of Andr\'asfai, Erd\H{o}s and S\'os \cite{AES} stating that any $n$-vertex triangle-free graph with minimum degree more than $2n/5$ must be bipartite,
Allen, Keevash, Sudakov, and Verstra\"ete \cite{AKSV} proved the following theorem, which answers their own question for extremal graphs satisfying appropriate minimum degree condition.

\begin{theorem}[Allen-Keevash-Sudakov-Verstra\"ete~\cite{AKSV}] \label{thm:AKSV-min}
Let $\F$ be an $(\alpha, \beta)$-smooth family with relative density $\rho$ and $2>\alpha>\beta\geq 1$.
Then there exists $k_0$ such that for any odd $k\geq k_0$ and sufficiently large $n$,
any $n$-vertex $\F\cup \{C_k\}$-free graph with minimum degree at least $\rho(\frac{2n}{5}+o(n))^{\alpha-1}$ is bipartite.
\end{theorem}

In this paper, we strengthen Theorem~\ref{thm:AKSV-min} by showing that the minimum degree condition can be lowered to $\delta n^{\alpha-1}$ for any given real $\delta>0$ and furthermore,
the condition on smoothness can be relaxed to the following notion.

\begin{definition} \label{dfn:nice}
Let $\alpha, \beta$ be reals with $2>\alpha>\beta\geq 1$. Let $\F$  be a family of bipartite graphs.
We say that $\F$ is {\it $(\alpha,\beta)$-quasi-smooth} with upper density $\rho$ and lower density $\rho_0$,
if there exist constants $\rho,\rho_0> 0$ and $C$ such that for all positive integers $m\leq n$,
\[z(m,n, \F)\leq \rho mn^{\alpha-1}+Cn^\beta \mbox{ and }
\ex(n,\F)\geq \rho_0 n^{\alpha}.\]
Note that the $n^\beta$ term becomes relevant when $m=o(n^{1+\beta-\alpha})$.
If $\F$ consists of a single graph $F$, then we just say that $F$ is $(\alpha,\beta)$-quasi-smooth.
\end{definition}


Clearly every $(\alpha,\beta)$-smooth graph is $(\alpha,\beta)$-quasi-smooth.
However, it is not known if every $(\alpha,\beta)$-quasi-smooth graph is $(\alpha,\beta)$-smooth.
For instance, it is proved that $\ex(n,K_{s,t})=\Omega(n^{2-1/s})$ for $t\geq (s-1)!+1$ in \cite{ARS,KRS} and for $t\geq C^s$ in a very recent paper of Bukh \cite{Bukh} (where $C$ is a constant).
Hence $K_{s,t}$ is quasi-smooth under these conditions, but it is unknown whether $K_{s,t}$ is always smooth.
The following is our main result in this paper.

\begin{theorem}\label{thm:main}
Let $\F$ be an $(\alpha,\beta)$-quasi-smooth family with $2>\alpha>\beta\geq 1$.
For any real $\delta>0$, there exists a positive integer $k_0$ such that
for any odd integer $k\geq k_0$ and sufficiently large $n$, any $n$-vertex $\F\cup\{C_k\}$-free graph with minimum degree at least $\delta n^{\alpha-1}$ is bipartite.
\end{theorem}

The proof of Theorem~\ref{thm:main} uses expansion properties  and a robust reachability lemma that is in part inspired by a lemma in a recent paper by Letzter \cite{Letzter} on the Tur\'an number of tight cycles.

As a direct application of Theorem~\ref{thm:main}, we also obtain the following strengthening of Theorem~\ref{thm:AKSV-smooth}.

\begin{theorem}\label{thm:improve}
Let $\F$ be an $(\alpha,\beta)$-smooth family with $2>\alpha>\beta\geq 1$.
Then there exists $k_0$ such that for any odd $k\geq k_0$ and sufficiently large $n$,
any $n$-vertex $\F\cup \{C_k\}$-free extremal graph can be made bipartite by deleting a set of $O(n^{1+\beta-\alpha})$ vertices, which together are incident to $O(n^\beta)$ edges.
Therefore, $\ex(n,\F\cup \{C_k\})=z(n,\F)+O(n^\beta)$.
\end{theorem}
Theorem~\ref{thm:improve} improves Theorem~\ref{thm:AKSV-smooth} in two ways. First, the error term is better. Second, the proof is  more concise and avoids the use of the sparse regularity lemma.
The theorem gives further evidence to an affirmative answer to the question of \cite{AKSV} that whether the extremal $n$-vertex $\F\cup \{C_k\}$-free graph $G$ in Theorem~\ref{thm:AKSV-smooth} is bipartite (for sufficiently large $n$).

We also prove an analogous theorem as Theorem~\ref{thm:main} for $C_{2\ell}$-free graphs, which complements the following result in Keevash-Sudakov-Verstra\"ete \cite{KSV}: For any integer $\ell\geq 2$, odd integer $k\geq 4\ell+1$ and any real $\delta>0$, the chromatic number of any $n$-vertex $\{C_4, C_6,...,C_{2\ell}, C_k\}$-free graph with minimum degree at least $\delta n^{1/\ell}$ is less than $(4k)^{\ell+1}/\delta^\ell$.

\begin{theorem} \label{thm:main2}
Let $\ell\geq 2$ be an integer.
For any real $\delta>0$, let $k_0=3\ell(8\ell/\delta)^\ell+2\ell+2$.
Then for any odd integer $k\geq k_0$ and sufficiently large $n$, any $n$-vertex
$\{C_{2\ell},C_k\}$-free graph with minimum degree at least $\delta n^{1/\ell}$ is bipartite.
\end{theorem}

This proof follows the same line as that of Theorem~\ref{thm:main}, except that
we will use a more efficient robust reachability lemma
for $C_{2\ell}$-free graphs and as a result get better control on $k_0$.

We should point out that the existence of such graphs in Theorem~\ref{thm:main2} is known only for $\ell\in \{2,3,5\}$ (see \cite{FS-survey}).
Also note that this result is not covered by Theorem~\ref{thm:main}, since $C_{2\ell}$ is not known to be $(\alpha,\beta)$-quasi-smooth for any $\ell\geq 3$.
In the concluding remarks, we will mention that Theorems~\ref{thm:main} and \ref{thm:main2} can
be extended to a slightly broader family of bipartite graphs that include both $(\alpha,\beta)$-quasi-smooth
graphs and $C_{2\ell}$'s.

The rest of the paper is organized as follows.
In  \Cref{sec:lemmas}, we develop some useful lemmas.
In \Cref{sec:c2ell}, we develop a lemma for $C_{2\ell}$-free graphs.
In \Cref{sec:main-proof}, we prove Theorem~\ref{thm:main} and Theorem~\ref{thm:main2}, respectively.
In \Cref{sec:improve}, we prove Theorem~\ref{thm:improve}.
In \Cref{sec:conclusion}, we give some concluding remarks.
Throughout this paper, we denote $[k]$ by the set $\{1,2,...,k\}$ for positive integers $k$.


\section{Some general lemmas} \label{sec:lemmas}

The main content of this section is to present key lemmas for our main result Theorem~\ref{thm:main}.

\begin{definition} \label{defn:ell0}
Let $\al,\be$ be reals with $2>\al>\be\geq 1$. Let $\ell_0(\al,\be)$ be defined as follows:
\[\ell_0=\left\lfloor \log_\be \frac{(2-\be)(\al-1)}{\al-\be}\right\rfloor+2, \mbox{ for } \beta>1
\quad \mbox{  and  } \quad \ell_0=\lfloor 1/(\al-1) \rfloor+1,  \mbox{ for } \beta=1.\]
\end{definition}

\begin{lemma}\label{lem:linear-size}
Let $\al,\be$ be reals with $2>\al>\be\geq 1$ and $\ell_0=\ell_0(\al,\be)$ be defined as in Definition~\ref{defn:ell0}.
Let $\F$ be an $(\alpha,\beta)$-quasi-smooth family of bipartite graphs
that satisfies $z(m,n,\F)\leq \rho mn^{\al-1}+Cn^\be$ for all $m\leq n$.
For any $\delta>0$, there exists a positive real $\mu=\mu(\alpha,\beta,\rho,\delta)$ such that
for all sufficiently large $n$ the following is true.
Let $G$ be an $\F$-free bipartite graph with at most $n$ vertices and minimum degree at least $\delta n^{\al-1}$.
Let $u\in V(G)$. For each $i\in \mathbb{N}$, let $N_i(u)$ denote the set of vertices at distance $i$ from $u$.
Then for some $j_0\leq \ell_0$ we have $\min\{|N_{j_0}(u)|,|N_{j_0+1}(u)|\}\geq \mu n$.
\end{lemma}
\begin{proof}
For each $i\in \mathbb{N}$, let $B_i$ denote the set of vertices at distance at most $i$ from $u$.
Let
\begin{equation}  \label{eq:mu-definition}
\gamma = (\delta/12\rho)^{1/(\al-1)} \mbox{ ~and ~}
\mu=\min\left\{(1/2) (\delta/2\rho)^{1/(\al-1)},
(\delta/4\rho)\gamma^{2-\al}, \gamma/\ell_0\right \}.
\end{equation}
First, we show that $|B_{\ell_0}|\geq \gamma n$.
Suppose for a contradiction that $|B_{\ell_0}|<\gamma n$.
Let $i\in [\ell_0-1]$.
Then clearly $|B_i|<\gamma n$, and since $G$ has minimum degree at least $\delta n^{\al-1}$, we have
\begin{equation} \label{eq:deg-sum1}
\sum_{v\in B_i} d(v)\geq \delta n^{\al-1} |B_i|.
\end{equation}
On the other hand, $\sum_{v\in B_i}d(v)=2e(B_i)+e(B_i,B_{i+1}\setminus B_i)$.
Since $G$ is bipartite and $\F$-free, $e(B_i)\leq \max_{(a,b)} \{\rho ab^{\al-1}+Cb^\be\}$ over
all pairs of positive integers $a\leq b$ with $a+b=|B_i|$.
Hence, $e(B_i)\leq \rho|B_i|^\al + C |B_i|^\be\leq 2\rho |B_i|^\al$, when $n$ is sufficiently large.
With some generosity, we can upper bound $e(B_i, B_{i+1}\setminus B_i)$ by $z(|B_i|, |B_{i+1}|,\F)$
to get $e(B_i, B_{i+1}\setminus B_i)\leq \rho |B_i||B_{i+1}|^{\al-1}+C |B_{i+1}|^\be$.
Putting the above estimations all together, we get
\begin{equation} \label{eq:deg-sum2}
\sum_{v\in B_i} d(v)\leq 4\rho |B_i|^\al + \rho |B_i||B_{i+1}|^{\al-1}+C|B_{i+1}|^\be.
\end{equation}
Combining~\eqref{eq:deg-sum1} and \eqref{eq:deg-sum2}, we get
\begin{equation}\label{eq:deg-comparision}
    \delta n^{\al-1}|B_i|\leq \sum_{v\in B_i} d(v)\leq 4\rho |B_i|^\al+\rho|B_i||B_{i+1}|^{\al-1}+C|B_{i+1}|^\be.
\end{equation}
If the first term on the right-hand side of~\eqref{eq:deg-comparision} is the largest term, then
we get $\delta n^{\al-1}|B_i|\leq 12\rho |B_i|^\al$, from which we get $|B_i|\geq (\delta/12\rho)^{1/(\al-1)} n\geq \gamma n$,
contradicting our assumption.
If the second term on the right-hand side of~\eqref{eq:deg-comparision} is the largest term, then
we get $\delta n^{\al-1}|B_i|\leq 3\rho |B_i||B_{i+1}|^{\al-1}$,
from which we get $|B_{i+1}|\geq (\delta/3\rho)^{1/(\al-1)} n\geq \gamma n$ and hence $|B_{\ell_0}|\geq |B_{i+1}|\geq \gamma n$, contradicting
our assumption. Hence we may assume that for each $i\in [\ell_0-1]$, we have
\[  \delta n^{\al-1}|B_i|\leq 3C|B_{i+1}|^\be,\]
which yields that for each $i\in [\ell_0-1]$,
\begin{equation}\label{eq:recurrence}
 |B_{i+1}|\geq (\delta/3C)^{1/\be} n^{(\al-1)/\be} |B_i|^{1/\be}.
 \end{equation}
Let $\{b_i\}$ be a sequence recursively defined by letting $b_1=\al-1$ and $b_{i+1}=(1/\be) b_i+ (\al-1)/\be$ for each $i\geq 1$.
If $\be=1$ then a closed form formula for $b_i$ is $b_i=(\al-1)i$. If $\be>1$ then a closed form formula
for $b_i$ is $b_i= \frac{\al-1}{\be-1} + (\frac{1}{\be})^{i-1}(\al-1-\frac{\al-1}{\be-1})$.
Note that we may assume $C\geq 1$, so $|B_1|\geq \delta n^{\al-1}\geq (\delta/3C)n^{b_1}$.
Then it follows by~\eqref{eq:recurrence} and induction that $|B_i|\geq (\delta/3C)^i n^{b_i}$ for each $i\in [\ell_0]$. However,
using the definition of $\ell_0$ we get $b_{\ell_0}>1$,
which yield $|B_{\ell_0}|>n$ as $n$ is sufficiently large.
This is a contradiction and thus proves that $|B_{\ell_0}|\geq \gamma n$.

Let $j\in [\ell_0]$ be the smallest index such that $|N_j|\geq (\gamma/\ell_0) n$. By the pigeonhole principle, such $j$ exists. By ~\eqref{eq:mu-definition}, $|N_j|\geq \mu n$.
Let $U=N_j$ and $V=N_{j-1}\cup N_{j+1}$. We show that $|V|\geq 2\mu n$, from which it follows
that either $|N_{j-1}|\geq \mu n$ or $|N_{j+1}|\geq \mu n$ and thus the lemma holds with $j_0=j$ or $j_0=j-1$.
Since all the edges of $G$ that are incident to $U$ are between $U$ and $V$, we have $e(G[U,V])\geq \delta n^{\al-1} |U|$.
On the other hand, $G[U,V]$ is $\F$-free. If $|V|\geq |U|$, then we have
$e(G[U,V])\leq \rho |U||V|^{\al-1}+ C|V|^\beta\leq 2\rho |U||V|^{\al-1}$,
where the last inequality holds as $|U|=|N_j|\geq \mu n$ and $n$ is sufficiently large.
Combining the two inequalities and solving for $|V|$, we get $|V|\geq (\delta/2\rho)^{1/(\al-1)} n\geq 2\mu n$, as desired.
Otherwise, we have $|V|\leq |U|$. Then $\delta n^{\al-1}|U|\leq e(G[U,V])\leq \rho |V||U|^{\al-1}+C|U|^\beta$.
Since $C|U|^\beta\ll\delta n^{\al-1}|U|$ for sufficiently large $n$,
we can derive from the above that $\delta n^{\al-1}|U|\leq 2\rho |V||U|^{\al-1}$.
Solving for $|V|$, we have $|V|\geq (\delta/2\rho)n^{\al-1} |U|^{2-\al}\geq (\delta/2\rho) n^{\al-1} (\gamma n)^{2-\al}=(\delta/2\rho)\gamma^{2-\al} n\geq 2\mu n,$
where the last inequality holds by \eqref{eq:mu-definition}, as desired.
\end{proof}

The following lemma, which we call {\it robust reachability lemma} is key to our proof of the main results.
It is inspired by a lemma used in
a recent paper of Letzter \cite{Letzter} on the Tur\'an number of tight cycles in hypergraphs.

\begin{lemma} \label{lem:robust-reachability}
Let $\al,\be$ be reals with $2>\al>\be\geq 1$ and $\ell_0=\ell_0(\al,\be)$ be defined as in Definition~\ref{defn:ell0}.
Let $\F$ be an $(\alpha,\beta)$-quasi-smooth family of bipartite graphs
that satisfies $z(m,n,\F)\leq \rho mn^{\al-1}+Cn^\be$ for all $m\leq n$.
Let $\mu = \mu(\al, \be, \rho, \delta)$ be defined as in Lemma~\ref{lem:linear-size}.
For any real $\delta>0$, the following holds for all sufficiently large $n$.
Let $G$ be an $\F$-free bipartite graph with at most $n$ verticers and minimum degree at least $\delta n^{\al-1}$.
Let $u\in V(G)$. Then there exists a set $S$ of at least $\mu (\al, \be, \rho, \delta/2) n$ vertices, and
a family $\Pa=\{P_v: v\in S\}$, where for each $v\in S$, $P_v$ is a $u,v$-path of length at most $\ell_0$, such that
no vertex except $u$ is used on more than $\frac{n}{\log n}$ of the paths in $\Pa$.
\end{lemma}
\begin{proof}
Let $S$ be a maximum set of vertices such that there is an associated family $\Pa=\{P_v:v\in S\}$,
where  for each $v\in S$, $P_v$ is a path of length at most $\ell_0$ such that no vertex is on more than $\frac{n}{\log n}$ of the paths.
Let $W$ denote the set of vertices in $G$ (other than $u$)  that lie on exactly $\frac{n}{\log n}$ of the paths $P_v$ in $\Pa$.
Then $|W|\frac{n}{\log n}\leq |S|\ell_0\leq n\ell_0$ and thus $|W|\leq \ell_0\log n<(\delta/2)n^{\al-1}$ for sufficiently large $n$.
Hence, $G-W$ has minimum degree at least $(\delta/2)n^{\al-1}$. If $|S|<\mu(\al, \be, \rho,\delta/2) n$, then by Lemma~\ref{lem:linear-size},
there exists a vertex $z\notin S$ and a $u,z$-path $P_z$ of length at most $\ell_0$ in $G-W$ that we can add to $S$ to contradict our choice of $S$.
Hence $|S|\geq \mu(\al, \be, \rho,\delta/2) n$.
\end{proof}

It is worth noting that the $\frac{n}{\log {n}}$ threshold could be improved to $O(n^{2-\alpha+\eps})$ for any real $\eps>0$, but for simplicity of presentation, we choose to use $\frac{n}{\log {n}}$ and such a choice suffices for the purpose of our main arguments.

The next folklore lemma will be used a few times and we include a proof for completeness. We would like to mention
that it might be easy for one to overlook
the connectedness of $H$ statement in the conclusion. But this condition will play important role in the main proofs.

\begin{lemma}\label{lem:bipartite}
Let $G$ be a connected graph. Let $H$ be a maximum spanning bipartite subgraph of $G$.
Then $H$ is connected and for each $v\in V(G), d_H(v)\geq (1/2) d_G(v)$.
\end{lemma}
\begin{proof}
Let $(X,Y)$ denote a bipartition of $H$. Suppose for contradiction that $H$ is disconnected
and $F$ is a component of $H$. Since $G$ is connected, it contains an edge $e$ joining
$V(F)$ to $V(G)\setminus V(F)$. But then $H\cup e$ is still bipartite, since adding $e$ does not create a new cycle.   Furthermore, $H\cup e$ has
more edges than $H$, contradicting our choice of $H$.

Next, let $v$ be any vertex in $H$. Without loss of generality, suppose $v\in X$.
Suppose $d_H(v)<(1/2)d_G(v)$. Then from $H$ by deleting the edges incident to $x$
and adding the edges in $G$ from $v$ to $X$, we obtained a bipartite subgraph of $G$ that has
more edges than $H$, a contradiction. Hence $\forall v\in V(G), d_H(v)\geq (1/2)d_G(v)$.
\end{proof}

We conclude this section with the following lemma about the diameter.
The {\it diameter} of a graph $G$ is the least integer $k$ such that there exists a path of length at most $k$ between any two vertices in $G$.

\begin{lemma}\label{lem:diameter}
Let $G$ be an $n$-vertex connected graph with minimum degree at least $D$. Then
$G$ has diameter at most $3n/D$.
\end{lemma}
\begin{proof}
Let $x,y$ be two vertices at maximum distance in $G$. Let $v_0v_1\cdots v_\ell$ be a shortest
$x,y$-path in $G$ where $v_0=x$ and $v_\ell=y$. Let $q=\lfloor \ell/3\rfloor$.
Note that $N(v_0), N(v_3), N(v_6),\cdots, N(v_{3q})$ are pairwise disjoint (or else we can find
a shorter $x,y$-path, a contradiction). Hence $n\geq \sum_{i=0}^q |N(v_{3i})| \geq (q+1)D$.
This implies that $(q+1)\leq n/D$ and hence $\ell\leq 3(q+1)\leq 3n/D$.
\end{proof}

\section{An efficient robust reachability lemma for $C_{2\ell}$-free graphs} \label{sec:c2ell}

In this section, we develop a more efficient robust reachability lemma than Lemma~\ref{lem:robust-reachability} for $C_{2\ell}$-free graphs, which may be of independent interest.
We need the following lemma from ~\cite{Verstraete}.

\begin{lemma}[Verstra\"ete \cite{Verstraete}] \label{lem:jacques}
Let $\ell\geq 2$ be an integer and $H$ a bipartite graph of average degree at least $4\ell$ and
girth $g$. Then there exist cycles of at least $(g/2-1)\ell\geq \ell$ consecutive even lengths in $H$. Moreover,
the shortest of these cycles has length at most twice the radius of $H$.
\end{lemma}

Our lemma is as follows.

\begin{lemma} \label{lem:C2ell-free}
Let $\ell\geq 2$ and $d$ be positive integers.
Let $H$ be a bipartite $C_{2\ell}$-free graph with
minimum degree at least $d$.  Let $u$ be any vertex in $H$.
Then the following items hold.
\begin{enumerate}
\item[(1).] The number of vertices that are at distance at most $\ell$ from $u$ is at least  $(d/4\ell)^\ell$.
\item[(2).]  Suppose $H$ has at most $n$ vertices and $d\geq 15\ell \log n$, where $n$ is sufficiently large.
Then
there is a set $S$ of at least $(1/2)(d/8\ell^2)^\ell$ vertices together with a family
$\Pa=\{P_v: v\in S\}$, where for each $v\in S$, $P_v$ is a $u,v$-path of exactly length $\ell$,
such that no vertex of $H$ except $u$ lies on more than $d^{\ell-1}$ of these paths and each vertex $v$
in $S$ lies only on $P_v$.
\end{enumerate}
\end{lemma}
\begin{proof} First we prove the first part (1) of the theorem.
Let $B_0=\{u\}$. Consider any $i\in [\ell]$.
Let $B_i$ denote the set of vertices at distance at most $i$ from $u$ in $H$
and $H_i$ the subgraph of $H$ induced by $B_i$.
If $H[B_i]$ has average degree at least $4\ell$, then by Lemma~\ref{lem:jacques}, $G_i$ contains cycles of $\ell$ consecutive even lengths
the shortest of which has length at most $2i\leq 2\ell$ and hence it contains $C_{2\ell}$, contradicting $G$ being $C_{2\ell}$-free.
So for each $i\in [\ell]$, we have $d(H_i)<4\ell$, which implies that
$e(H_i)< 2\ell |B_i|$. On the other hand, $H_i$ contains all the edges of $G$ that are incident to $B_{i-1}$. So $e(H_i)\geq d|B_{i-1}|/2$.
Combining these two inequalities, we get $ 2\ell|B_i|>d|B_{i-1}|/2$. Hence, $|B_i|>(d/4\ell) |B_{i-1}|$ for each $i\in [\ell]$.
Thus, $|B_\ell|\geq (d/4\ell)^\ell$, as desired.

Next, we prove the second part (2).
Let us randomly split the vertices of $G$ into $\ell$ parts $V_1,\dots, V_\ell$. For each vertex $x$, and
each $i\in [\ell]$, the degree $d_i(x)$ of $x$ in $V_i$ has a binomial distribution $\Bin(d(x), 1/\ell)$. Hence, using Chernoff's inequality
(see \cite{Alon-Spencer} or \cite{JLR} Corollary 2.3), we have
\[\mathbb{P} [d_i(x)< (1/2\ell) d(x)]\leq \mathbb{P}[|d_i(x)-(1/\ell) d(x)|>(1/2\ell) d(x)]\leq 2e^{-d(x)/12\ell} =o(n^{-1}),\]
since $d(x)\geq 15\ell \log n$. Hence, for sufficiently large $n$, there exists a splitting of $V(H)$ such that for each $x\in V(H)$ and
for each $i\in [\ell]$, $d_i(x)\geq (1/2\ell) d(x)\geq d/2\ell$. Now, we form a subgraph $H'$ of $H$ as follows.
First, we include exactly $d/2\ell$ of the edges from $u$ to $V_1$. Denote the set of reached vertices in $V_1$ by $S_1$.
Then for each vertex in $S_1$ including exactly $d /2\ell$ edges from it to $V_2$. Denote the set of reached vertices in $V_2$ by $S_2$.
We continue like this till we define $S_\ell$.
Let $B_0=S_0=\{u\}$.
For each $i\in [\ell]$, let $B_i=\bigcup_{j=0}^i S_i$.
and $H_i$ the subgraph of $H'$ induced by $B_i$. Note that $H_i$ has radius $i$.
As in the proof of the first part of the lemma, since $H_i$ is $C_{2\ell}$-free, $e(H_i)< 2\ell |B_i|$.
On the other hand, $H_i$ contains all the edges of $H'$ that are incident to $B_{i-1}$. So $e(H_i)\geq (1/2)(d/2\ell)|B_{i-1}|$.
Combining these two inequalities, we get $ 2\ell|B_i|>(d/4\ell)|B_{i-1}|$.
Hence, $|B_i|>(d/8\ell^2) |B_{i-1}|$ for each $i\in [\ell]$. Thus, $|B_\ell|\geq (d/8\ell^2)^\ell$.

It is easy to see that $\sum_{i=0}^{\ell-1} |S_i|\leq \sum_{i=0}^{\ell-1} (d/2\ell)^i\leq 2(d/2\ell)^{\ell-1}$, when $n$ is sufficiently large.
Hence
\[|S_\ell|=|B_\ell\backslash \cup_{i=0}^{\ell-1}S_i|\geq (d/8\ell^2)^\ell - 2 (d/2\ell)^{\ell-1}>(1/2) (d/8\ell^2)^\ell,\]
where $n$ (and thus $d$) is sufficiently large.
By the definition of $H'$, for each $v\in S_\ell$, there is a path of length $\ell$ from $u$ to $v$ that
intersects each of  $V_1,V_2,\dots, V_\ell$. From the union of these paths one can find a tree $T$ of height $\ell$
rooted at $u$, in which all the vertices in $S_\ell$ are at distance $\ell$ from $u$. Furthermore, by
the definition of $H'$, $T$ has maximum degree at most $(d/2\ell)+1$. For each $v\in S_\ell$, let
$P_v$ be the unique $u,v$-path in $T$. If $x$ is any vertex in $T$ other than $u$, then clearly
$x$ lies on at most $(d/2\ell)^{\ell-1}$ of the paths $P_v$.
Furthermore, each $v\in S_\ell$ doesn't lie on any $P_{w}$ for $w\in S_\ell\setminus \{v\}$.
\end{proof}

\section{Proofs of Theorem~\ref{thm:main} and Theorem~\ref{thm:main2}}\label{sec:main-proof}

Even though the proofs of Theorem~\ref{thm:main} and Theorem~\ref{thm:main2} are essentially the same, there are sufficiently different choices of parameters that we will prove them separately. Before giving the formal proof of Theorem~\ref{thm:main}, we give an overview.
Let $\al,\be$ be reals with $2>\al>\be\geq 1$.
Let $\F$ be an $(\alpha,\beta)$-quasi-smooth family of bipartite graphs
that satisfies $z(m,n,\F)\leq \rho mn^{\al-1}+Cn^\be$ for all $m\leq n$.
Let $\delta>0$ be given. We wish to show that there exists a positive integer $k_0$ such that
for any odd integer $k\geq k_0$ and sufficiently large $n$, any $n$-vertex $\F\cup\{C_k\}$-free graph with minimum degree at least $\delta n^{\alpha-1}$ is bipartite.

Let $k_0$ be sufficiently large as a function of $\alpha,\beta,\rho$, and $\delta$. Let $k\geq k_0$
be an odd integer. Let $G$ be a $n$-vertex
$\F\cup\{C_k\}$-free graph with minimum degree at least $\delta n^{\alpha-1}$. We take a maximum spanning
bipartite subgraph $H$ of $G$. Let $(X,Y)$ be a bipartition of $H$. We show that $G$ is itself bipartite by
showing that $X,Y$ must be independent sets in $G$. Suppose contradiction there exist two vertices $u,v$, say in $X$,
such that $uv\in E(G)$. We derive a contradiction by finding a $C_k$ in $G$ that contains $uv$. To build such a $C_k$,
we utilize expansion properties (as described in Lemma~\ref{lem:linear-size}) and robust reachability properties (as described in Lemma ~\ref{lem:robust-reachability}) of various carefully defined subgraphs of $G$.
First, via a random partitioning argument, we can find a partition of $V(H)=V(G)$ into two subsets $A,B$ such that $H[A]$ has high minimum degree
and each vertex in $A$ has high degree in $B$ (inside $H$). Assume $u\in A$. We then apply Lemma~\ref{lem:robust-reachability} and some additional cleaning
to find a balanced family $\Pa'$ of paths of bounded equal length inside $H[A]$ that start at $u$ and reach a linear-sized subset $S''$ of $A$ with
the additional property that vertices in $S''$ serve only as endpoints of these paths and never as interior points. Since vertices in $S''$
have high degree in $H$ into $B$ and $S''$ is linear-sized, the subgraph of $H$ consisting of edges from $S''$ to $B$ is dense and
contains a subgraph $H''$ of high minimum degree. We then take a shortest path $Q$ in $G$ from $v$ to $V(H'')$ and denote its unique vertex
in $V(H'')$ by $y$. The balanced-ness of $\Pa'$ ensures that most of the paths in $\Pa'$ reaching $S''$ are vertex disjoint from $Q$.
We then apply Lemma~\ref{lem:linear-size} inside $H''$ along with some additional cleaning to find a path in $H''$ of suitable length from $y$
to an appropriate vertex $w^*$ in $S''$. We build a $C_k$ by taking the union of this path with $Q$, the edge $uv$ and the member of $\Pa'$ from $u$
to $w^*$.

\begin{proof}[Proof of Theorem~\ref{thm:main}]
Let $\al,\be$ be reals with $2>\al>\be\geq 1$.
Let $\F$ be an $(\alpha,\beta)$-quasi-smooth family of bipartite graphs
that satisfies $z(m,n,\F)\leq \rho mn^{\al-1}+Cn^\be$ for all $m\leq n$.

Given any real $\delta>0$, we first define $k_0$ as following.
Let $\ell_0=\ell_0(\al,\be)$ as in  Definition~\ref{defn:ell0}.  Let $\mu(\delta) = \mu(\al, \be, \rho, \delta)$ as in the proof of  Lemma~\ref{lem:linear-size}.
Define
\begin{equation} \label{eq:L-definition}
L:=\left\lfloor\frac{3}{\mu(\delta/2)}\right \rfloor \cdot \ell_0 \mbox{ ~and~ } k_0:= 2\ell_0+L+2.
\end{equation}
Let $k\geq k_0$ be odd.
Let $n$ be sufficiently large so that all subsequent
inequalities involving $n$ hold.
Let $G$ be an $n$-vertex $\F\cup \{C_k\}$-free graph with minimum degree at least $\delta n^{\al-1}$.
We may assume that $G$ is connected.
Let $H$ be a maximum bipartite spanning subgraph of $G$. By Lemma ~\ref{lem:bipartite},
$H$ is connected and has minimum degree at least $(\delta/2) n^{\al-1}$.

Since $H$ is $\F$-free, by Lemma~\ref{lem:linear-size}, for each vertex $x$, the set of vertices that are at distance at most $\ell_0$ from $x$ is at least $\mu(\delta/2) n$. Hence by Lemma~\ref{lem:diameter}
(applied to the $\ell_0$-th power $H^{\ell_0}$ of $H$), $H^{\ell_0}$ has diameter at most
$\lfloor \frac{3n}{\mu(\delta/2) n}\rfloor
=\lfloor\frac{3}{\mu(\delta/2)}\rfloor$ and hence $H$ has diameter at most
$\lfloor\frac{3}{\mu(\delta/2)}\rfloor\cdot \ell_0=L$, as defined in~\eqref{eq:L-definition}.

Let $(X,Y)$ be the unique bipartition of $H$. We show that $G$ is also bipartite with $(X,Y)$ being
a bipartition of it. Suppose otherwise.
We may assume, without loss of generality, that there exist two vertices $u,v\in X$ such that $uv\in E(G)$. We will derive a contradiction by finding a copy of $C_k$ in $G$ that contains $uv$.

Let us randomly split $V(H)$ into two subsets $A,B$. For each vertex $x$ of degree $d(x)$ in $H$,
let $d_A(x)$ and $d_B(x)$ denote the degree of $x$ in $A$ and $B$, respectively.
Then both $d_A(x)$ and $d_B(x)$ satisfy the binomial distribution $\Bin(d(x), 1/2)$. Hence,
by Chernoff's inequality,  we have\[\mathbb{P}(d_A(x)< (1/4)d(x))\leq \mathbb{P}(|d_A(x)-d(x)/2|\geq d(x)/4)\leq 2e^{-d(x)/24} =o(n^{-1}),\]
since $d(x)\geq (\delta/2)n^{\al-1}$ and $n$ is sufficiently large.
Hence with positive probability we can ensure that for any $x\in V(H)$, $\min\{d_A(x), d_B(x)\}\geq (1/4)d(x)\geq (\delta/8) n^{\al-1}$.
Let us fix such a partition $A,B$ of $V(H)$.

Without loss of generality, suppose that the vertex $u$ is in $A$.
Let $H[A]$ denote the subgraph of $H$ induced by $A$. By our discussion above, $H[A]$ has minimum degree at least $(\delta/8) n^{\al-1}$. By Lemma~\ref{lem:robust-reachability}, there exists
a set $U$ of at least $\mu(\delta/16) n$ vertices and a family $\Pa=\{P_z: z\in U\}$, where
for each $z\in U$, $P_z$ is a $u,z$-path in $H[A]$ of length at most $\ell_0$ and no vertex in $H$ lies on more than $n/\log n$ of these paths $P_z$.
By the pigeonhole principle, there exists a value $p\in [\ell_0]$ and a subset $S\subseteq U$ of size
\[|S|\geq |U|/\ell_0\geq (\mu(\delta/16)/\ell_0) n\] such that for each $z\in S$, $P_z \in \Pa$ has length $p$.
Now, let us randomly color the vertices in $H[A]$ with colors $1$ and $2$.
For each $z\in S$, the path $P_z$ is {\it good} if $z$ is colored $2$ and the other $p$ vertices on $P_z$ are colored $1$.
The probability that $P_z$ is good is $1/2^{p+1}$.
Hence, for some coloring, there are at least $|S|/2^{p+1}$ good paths. Let $S'$ denote the subset of vertices $z\in S$
such that $P_z$ is good and let $\Pa'=\{P_z:z\in S'\}$. By our discussion,
\[|\Pa'|=|S'|\geq \left(\mu(\delta/16)/(2^{\ell_0+1}\ell_0)\right)n.\]
Note that by our definition of $\Pa'$, no vertex in $S'$ is used as an internal vertex of any path in $\Pa'$.
By our earlier discussion, each vertex in $S'$ has at least $(\delta/8)n^{\al-1}$ neighbors in $B$. Let $H'$ be the subgraph of $H$ whose edge set contains all the edges in $H$ from $S'$ to $B$. Therefore,  $H'$ is bipartite with partition $S'$ and $B$. Furthermore,
\[e(H')\geq |S'|(\delta/8) n^{\al-1}\geq \left(\mu(\delta/16)\delta/(2^{\ell_0+4}\ell_0)\right) n^\al =\gamma n^\al,\]
where $\gamma:=\gamma(\delta)=\mu(\delta/16) \delta/(2^{\ell_0+4}\ell_0)$.
So $H'$ has average degree at least $2\gamma n^{\al-1}$.
By a well-known fact, $H'$ contains a subgraph $H''$ of minimum degree at least $\gamma n^{\al-1}$.
Let $S''=V(H'')\cap S'$.

Let $Q$ be a shortest path in $H$ from the vertex $v$ to $V(H'')$. Let $y$ be the endpoint of $Q$ in $V(H'')$.
By our choice of $Q$, $y$ is the only vertex in $V(Q)\cap V(H'')$ (note that it is possible that $y=v$).
Let $q$ denote the length of $Q$. Since $H$ has diameter at most $L$, we have $q\leq L$.

By Lemma~\ref{lem:linear-size}, for some $j_0\leq \ell_0$, inside the graph $H''$ we have $\min\{|N_{j_0}(y)|,|N_{j_0+1}(y)|\}\geq \mu( \gamma) n$.
Note that one of $N_{j_0}(y)$ and $N_{j_0+1}(y)$ lies completely
inside $V(H'')\cap A$. Denote this set by $W$. Let $$W_0=\{w\in W: P_w\cap V(Q)\neq \emptyset\}.$$
Since $Q$ contains at most $L+1$ vertices, each of which lies on at most $n/\log n$ of the members of $\Pa$, we see $|W_0|\leq (L+1)(n/\log n)$.
Since $n$ is sufficiently large, we have
\[|W\setminus W_0|\geq \mu(\gamma) n - (L+1)n/\log n\geq (1/2)\mu(\gamma) n.\]
Let $H[W\setminus W_0,B]$ denote the subgraph of $H$ consisting of edges that have one endpoint in $W\setminus W_0$ and the other endpoint in $B$.
Since each vertex in $W\setminus W_0$ has at least $(\delta/8) n^{\al-1}$ neighbors in $B$,
\[e(H[W\setminus W_0, B])\geq |W\setminus W_0| (\delta/8) n^{\al-1}
\geq (1/16)\mu(\gamma)  \delta n^\al.\]
Hence $H[W\setminus W_0,B]$ contains a subgraph $H^*$ with minimum degree at least $(1/16)\mu(\gamma)\delta n^{\al-1} \geq k$, for sufficiently large $n$.
Let $w$ be any vertex in $V(H^*)\cap (W\setminus W_0)$.
By our definition of $W$, there is a path $R_w$ in $H''$ of length
$r\leq j_0+1\leq \ell_0+1$ from $y$ to $w$. Let $t=k-1-q-r-p$.
Since $q\leq L, p\leq \ell_0, r\leq \ell_0+1$ and $k\geq k_0= 2\ell_0+L+2$, we get $t\geq 0$.
Since there is a path in $H$ of length $p$ from $w$ to $u$ (by the definition of $S''\subseteq W$). Since $V(Q)\cap V(H'')=\{y\}$, $R_w\cup Q$ is a path in $H$
 of length $q+r$ from $w$ to $v$ and $u,v\in X$, $p$ and $q+r$ have the same parity. Since $k$ is odd, we see that $t$ is even.
 Since $H^*$ has minimum (much) larger than $k$, greedily we can build a path $T$ of length $t$ in $H^*$ from $w$ to some vertex $w^*$
 in $V(H^*)\cap (W\setminus W_0)$ such that $T$ intersects $Q\cup R_w$ only in $w$.  Now, let
 \[C:=uv\cup Q\cup R_w\cup T\cup P_{w^*}\]
By our definitions, $Q\cup R_w\cup T$ is a path. Also, since $w^*\in W\setminus W_0$,
$P_{w^*}$ is vertex disjoint from $Q$. Finally, by our definition of $\Pa'$, $V(P_{w^*})\setminus \{w^*\}$ is
disjoint from $S'$ and hence from $S''$. It is certainly also disjoint from $B$ and hence is
vertex disjoint from $R_w\cup T$. So, $C$ is a cycle in $G$ of length $1+q+r+t+p=k$, a contradiction.
This proves Theorem~\ref{thm:main}.
\end{proof}

\bigskip

\begin{proof} [Proof of Theorem~\ref{thm:main2}]
Let $\ell\geq 2$ be an integer. Let $\delta>0$ be a real. Define
\begin{equation} \label{eq:L-definition2}
L:=3\ell (8\ell/\delta)^{\ell} \mbox{ ~and~ } k_0:= 2\ell+L+2.
\end{equation}
Let $k\geq k_0$ be an odd integer.
Let $n$ be sufficiently large so that all subsequent
inequalities involving $n$ hold.
Let $G$ be an $n$-vertex $C_{2\ell}$-free graph with minimum degree at least $\delta n^{1/\ell}$.
We may assume that $G$ is connected.
Let $H$ be a maximum spanning subgraph of $G$. By Lemma~\ref{lem:bipartite},
$H$ is connected with minimum degree at least $(\delta/2) n^{1/\ell}$.
Since $H$ is $C_{2\ell}$-free,
by Lemma~\ref{lem:C2ell-free}, the $\ell$-th power $H^{\ell}$ of $H$ has minimum degree at least
$(\delta n^{1/\ell}/8\ell)^{\ell} =(\delta/8\ell)^{\ell} n$. Hence, by Lemma~\ref{lem:diameter},
$H^\ell$ has diameter at most $3n/[(\delta/8\ell)^{\ell} n]=3 (8\ell/\delta)^{\ell}$. Therefore, $H$ has
diameter at most $3\ell (8\ell/\delta)^{\ell}= L$, as defined in~\eqref{eq:L-definition2}.

Let $(X,Y)$ the unique bipartition of $H$. We show that $G$ is also bipartite with $(X,Y)$ being
a bipartition of it. Suppose otherwise. Then without loss of generality, we may assume that there exist two vertices $u,v\in X$ such that $uv\in E(G)$. We will derive a contradiction by finding a copy of $C_k$ in $G$ that contains $uv$.
As in the proof of Theorem~\ref{thm:main}, we can
split $V(H)$ into two subsets $A,B$ such that for
each vertex
$x\in V(H)$, we have $d_A(x), d_B(x)\geq (1/4)d(x)\geq (\delta/8) n^{1/\ell}$.

Without loss of generality, suppose $u\in A$. Let $H[A]$ denote the subgraph of $H$ induced by $A$. By our discussion above, $H[A]$ has minimum degree at least $(\delta/8) n^{1/\ell}$. By Lemma~\ref{lem:C2ell-free} (with $d=(\delta/8) n^{1/\ell}$), there exists
a set $S$ of size at least $(\delta/64\ell^2)^\ell (n/2)$ and a family $\Pa=\{P_z: z\in S\}$, where
for each $z\in S$, $P_z$ is a $u,z$-path in $H[A]$ of length $\ell$, such that no vertex other than $u$ in $H$ lies on
more than $(\delta n^{1/\ell}/8)^{\ell-1}=(\delta/8)^{\ell-1} n^{1-1/\ell}$ of these paths. Furthermore, for each $z\in S$,
$z$ lies only on $P_z$.

Let $H[S,B]$ denote the bipartite subgraph of $H$ induced by the two parts $S$ and $B$.
By our earlier discussion, each vertex in $S$ has at least $(\delta/8)n^{1/\ell}$ neighbors in $B$. Hence,
\[e(H)\geq |S|(\delta/8) n^{1/\ell}\geq (\delta^{\ell+1}/2^{6\ell+4}\ell^{2\ell}) n^{1+1/\ell}=\gamma n^{1+1\ell},\]
where $\gamma:=\delta^{\ell+1}/2^{6\ell+4}\ell^{2\ell}$.
Then $H[S,B]$ has average degree at least $2\gamma n^{1/\ell}$ and thus contains a subgraph $H'$ of minimum degree at least $\gamma n^{1/\ell}$.
Let $S'=V(H')\cap S$ and $B'=V(H')\cap B$.

If $\ell$ is even, then $S'\subseteq X$ and let $Q$ be a shortest path in $H$ from $v$ to $S'$.
If $\ell$ is odd, then $B'\subseteq X$ and let $Q$ be a shortest path in $H$ from $v$ to $B'$.
Let $y$ denote the endpoint of $Q$ opposing $v$ (it is possible that $y=v$).
Let $q$ denote the length of $Q$.
So $q$ is even and $y\in X$.
In either case, it is easy to see that $q\leq L+1$
and that $V(H')$ contains $y$ and at most one other vertex of $Q$.
Hence $H'-(V(Q)\setminus \{y\})$ has minimum degree at least
$\gamma n^{1/\ell}-1\geq (\gamma/2)n^{1/\ell}$.
By Lemma~\ref{lem:C2ell-free} (with $d=(\gamma/2)n^{1/\ell}$), inside the graph $H'-(V(Q)\setminus \{y\})$  there is a set $W$ of size at least $(1/2)(\gamma/16\ell^2)^\ell n$
such that for each $w\in W$ there is a path $R_w$ of length $\ell$ from $y$ to $w$ in $H'-(V(Q)\setminus \{y\})$.
Furthermore, by our definition of $Q$, we can get $W\subseteq S'$.
Recall the paths $P_w$ in $\Pa$.
Let $$W_0=\{w\in W: P_w\cap V(Q)\neq \emptyset\}.$$
Since $Q$ contains at most $L+1$ vertices each of which lies on at most $(\delta/8)^{\ell-1} n^{1-1/\ell}$
of the members of $\Pa$, we see $|W_0|\leq (L+1)(\delta/8)^{\ell-1} n^{1-1/\ell}$.
Since $n$ is sufficiently large, we have
\[|W\setminus W_0|\geq (1/2)(\gamma/16\ell^2)^\ell n -  (L+1)(\delta/8)^{\ell-1} n^{1-1/\ell} \geq
(1/4)(\gamma/16\ell^2)^\ell n.\]
Since each vertex in $W\setminus W_0$ has at least $(\delta/8) n^\ell$ neighbors in $B$,
\[e(H[W\setminus W_0, B])\geq |W\setminus W_0| (\delta/8) n^{1/\ell}\geq (\delta\cdot \gamma^\ell)/(2^{4\ell+5}\cdot \ell^{2\ell}) n^{1+1/\ell}.\]
Hence $H[W\setminus W_0,B]$ contains a subgraph $H^*$ with minimum degree at least
$(\delta\cdot \gamma^\ell)/(2^{4\ell+5}\cdot \ell^{2\ell}) n^{1/\ell}\geq k$, for sufficiently large $n$.
Let $w$ be any vertex in $V(H^*)\cap (W\setminus W_0)$.
By our definition of $W$, there is a path $R_w$ in $H'-(V(Q)\setminus \{y\})$ of length
$\ell$ from $y$ to $w$. Let $t=k-1-q-2\ell$. Since $q\leq L+1$ and $k\geq k_0= 2\ell+L+2$, we see $t\geq 0$.
Since $k$ is odd and $q$ is even, we also see that $t$ is even.
Since $H^*$ has minimum degree (much) larger than $k$, greedily we can build a path $T$ of length $t$ in $H^*$ from $w$ to some vertex $w^*$
 in $V(H^*)\cap (W\setminus W_0)$ such that $T$ intersects $Q\cup R_w$ only in $w$.  Now, let
 \[C:=uv\cup Q\cup R_w\cup T\cup P_{w^*}\]
By our definition of $T$, $Q\cup R_w\cup T$ is path. Also, since $w^*\in W\setminus W_0$,
$P_{w^*}$ is vertex disjoint from $Q$. Finally $P_{w^*}\setminus \{w^*\}$ does not contain any vertex of $S'\cup B$ and
hence is vertex disjoint from $R_w\cup T$. Hence, $C$ is a cycle in $G$ of length $1+q+2\ell+t=k$, a contradiction.
\end{proof}


\section{Proof of Theorem~\ref{thm:improve}} \label{sec:improve}

\begin{proof}[Proof of Theorem~\ref{thm:improve}]
Let $2>\alpha>\beta\geq 1$ and $\F$ be an {\it $(\alpha,\beta)$-smooth} family with relative density $\rho$.
Then by the remark after Definition~\ref{dfn:smooth}, there exist constants $C_1<C_2$ such that for sufficiently large  $n$, $$\rho(n/2)^\alpha+C_1n^\beta\leq z(n,\F)\leq \rho(n/2)^\alpha+C_2n^\beta.$$
Fix $\delta:=\rho/2^{\alpha+3}$.
Let $k_0$ be from Theorem~\ref{thm:main} such that for any odd $k\geq k_0$ and sufficiently large $m$, any $m$-vertex $\F\cup\{C_k\}$-free graph with minimum degree at least $\delta m^{\alpha-1}$ is bipartite.

Now consider any odd $k\geq k_0$ and sufficiently large $n$. Let $G$ be an $n$-vertex extremal $\F\cup \{C_k\}$-free graph.
Then
\begin{equation}\label{eq:O(beta)-1}
e(G)=\ex(n,\F\cup \{C_k\})\geq z(n,\F)\geq \frac{\rho}{2^\alpha}n^\alpha+C_1n^\beta.
\end{equation}
Let $G_0=G$.
If there exists some vertex $x$ of degree less than $\delta n^{\alpha-1}$ in $G_0$, then we delete the vertex $x$ and rename the remaining subgraph as $G_0$.
We repeat the above process until there is no such vertex in $G_0$. Let
 $H$ denote the remaining induced subgraph of $G$ and let $t=n-|V(H)|$.
We note that as $\alpha<2$ and $n$ is sufficiently large, using \eqref{eq:O(beta)-1} and $\delta=\frac{\rho}{2^{\al+3}}$,
\begin{equation}\label{eq:O(beta)-2}
e(H)\geq e(G)-t\cdot \delta n^{\alpha-1}\geq \left(\frac{\rho}{2^\alpha}n^\alpha+C_1n^\beta\right)-n\cdot \delta n^{\alpha-1}=\frac{7}{8}\frac{\rho}{2^\alpha} n^\alpha +C_1n^\beta.
\end{equation}
Let $m=|V(H)|$.
By definition, $H$ is an $m$-vertex $\F\cup\{C_k\}$-free graph with minimum degree at least $\delta n^{\alpha-1}\geq \delta m^{\alpha-1}$. By taking $n$ sufficiently large, we can make $m$ large enough to apply Theorem~\ref{thm:main} to
conclude that $H$ is bipartite. Since $H$ is also, $\F$-free, we have
\begin{equation} \label{eq:H-upper}
e(H)\leq z(m,\F)=z(n-t,\F)\leq \frac{\rho}{2^\alpha} (n-t)^\alpha +C_2 n^\beta.
\end{equation}
Comparing \eqref{eq:O(beta)-2} and ~\eqref{eq:H-upper}, we see that $n-t\geq \frac{3}{4}n$. By
\eqref{eq:O(beta)-1}, \eqref{eq:O(beta)-2} and ~\eqref{eq:H-upper}, we have
$$\left(\frac{\rho}{2^\alpha}n^\alpha+C_1n^\beta\right)-t\cdot \delta n^{\alpha-1}\leq \frac{\rho}{2^\alpha}(n-t)^\alpha+C_2 n^\beta
$$

Recall that $\delta=\rho/2^{\alpha+3}$.
Rearranging the above inequality, we get

$$\frac{\rho}{2^\alpha}n^\alpha-\frac{\rho}{2^\alpha}(n-t)^\alpha-t\cdot \delta n^{\alpha-1}\leq (C_2 - C_1)n^\beta. $$
By the Mean Value Theorem, for some $n - t \leq n' \leq n$, this is equivalent to  $$\frac{\rho\alpha}{2^\alpha}t(n')^{\alpha -1}-t\cdot \delta n^{\alpha-1}\leq (C_2 - C_1)n^\beta.$$ Since $n-t\geq \frac{3}{4}n$ and
$\frac{\rho\alpha}{2^{\alpha}}(\frac{3}{4})^{\alpha -1} > 2\cdot \frac{\rho}{2^{\alpha+3}}$, the above inequality yields
\[\frac{\rho}{2^{\alpha+3}}tn^{\alpha - 1} \leq (C_2 - C_1)n^\beta.\]
So, $t=O(n^{1+\beta-\alpha})$ and
in obtaining $H$ from $G$ at most $t\cdot \delta n^{\alpha-1}=O(n^\beta)$ edges are removed.
\end{proof}

\section{Concluding Remarks} \label{sec:conclusion}

\begin{enumerate}
    \item As mentioned in the introduction, there are bipartite graphs that are not smooth. We give such an example here.
    Given integers $t,\ell\geq 2$, the {\it theta graph}
    $\theta_{t,\ell}$ is the graph consisting of $t$ internally disjoint paths of length $\ell$ between two vertices.
    In particular, we have $\theta_{2,\ell}=C_{2\ell}$.
    Faudree and Simonovits \cite{FS} showed that for all $t,\ell\geq 2$, $\ex(n,\theta_{t,\ell})=O(n^{1+1/\ell})$ (the case $t=2$ was first proved by Bondy-Simonovits \cite{BS}).
    Conlon~\cite{conlon-theta} showed that for each $\ell\geq 2$, there exists a $t_0$ such that for all $t\geq t_0$, $\ex(n,\theta_{t,\ell})=\Omega(n^{1+1/\ell})$,
    the leading coefficients of which were further improved by Bukh-Tait \cite{BT}.
    Jiang, Ma and Yepremyan~\cite{JMY} showed that
    for all $t,\ell$, there exists a constant $c=c(t,\ell)$ such that
    for all $m\leq n$
    \begin{equation*}
    z(m,n,\theta_{t,\ell})\le \left\{\begin{array}{ll}
    c\cdot [(mn)^{\frac{\ell+1}{2\ell}}+m+n] & \text{if } \ell \text{  is odd}, \\
    c\cdot [m^{\frac{1}{2}+\frac{1}{\ell}}n^{\frac{1}{2}}+m+n] & \text{if } \ell \text{ is even}.
  \end{array}\right.
\end{equation*}
For the case $t=2$, the above bound was first proved by Naor-Verstra\"ete \cite{NV}, and a different form of the upper bound on $z(m,n,\theta_{2,\ell})$ was obtained by Jiang-Ma \cite{JM}.
On the other hand, using first moment deletion method it is not hard to show that
\begin{proposition}\label{prop:theta-lower}
Let $\eps>0$ be any real. Let $\ell\geq 2$. There exists a $t_0$ such that for all $t\geq t_0$, if $\ell$ is odd then
\[z(m,n,\theta_{t,\ell})\geq \Omega(m^{\frac{\ell+1}{2\ell}-\eps} n^{\frac{\ell+1}{2\ell}-\eps}),\]
and if $\ell$ is even then
\[z(m,n,\theta_{t,\ell})\geq \Omega(m^{\frac{1}{2}+\frac{1}{\ell}-\eps} n^{\frac{1}{2}-\eps}).\]
\end{proposition}
\begin{proof}
Consider the bipartite random graph $G\in G(m,n,p)$ with $p$ to be chosen later.
Let $q=\lfloor \ell/2\rfloor$.
Let $X=e(G)$ and $Y$ denote the number of copies of $\theta_{t,k}$ in $G$.
We have $\mathbb{E}(X)=mnp$. If $\ell$ is odd, then $\ell=2q+1$ and $\mathbb{E}[Y]\leq [m]_{tq+1}[n]_{tq+1} p^{t(2q+1)}<(1/2)m^{tq+1}n^{tq+1}p^{t(2q+1)}$. We now choose $p$ so that $\mathbb{E}[X]\geq 2\mathbb{E}[Y]$.
It suffices to set $p=m^{\frac{-tq}{2tq+t-1}}n^{\frac{-tq}{2tq+t-1}}$. Since $\mathbb{E}(X-Y)\geq (1/2)\mathbb{E}[X]$, there exists a $(m,n)$-bipartite graph $G$ for which $X-Y\geq (1/2) mnp
=(1/2)(mn)^{\frac{tq+t-1}{2tq+t-1}}$. By deleting one edge from copy of $\theta_{t,\ell}$ in $G$, we obtained
a $(m,n)$-bipartite graph $G'$ that is $\theta_{t,\ell}$-free and satisfies
\[e(G')\geq (1/2) (mn)^{\frac{tq+t-1}{2tq+t-1}}=(1/2)(mn)^{\frac{\ell+1-(2/t)}{2\ell-(2/t)}}. \]
For sufficiently large $t$, we have $e(G')\geq (1/2) (mn)^{\frac{\ell+1}{2\ell}-\eps}$, as desired.

For even integers $\ell=2q$, the analysis is similar, except that we use the bound
$\mathbb{E}[Y]\leq [m]_{tq+1}[n]_{t(q-1)+1}+[m]_{t(q-1)+1}[n]_{tq+1}) p^{2tq}<(1/2)m^{t(q-1)+1}n^{tq+1}p^{2tq}$.
We omit the details.
\end{proof}

It is quite likely using the random algebraic method used in \cite{conlon-theta}, one could show that
for each $\ell\geq 2$, there exist a $t_0$ such that for all $t\geq t_0$ if $\ell$ is odd then
$z(m,n,\theta_{t,\ell})\geq \Omega(m^{\frac{\ell+1}{2\ell}} n^{\frac{\ell+1}{2\ell}})$ and
if $\ell$ is even then $z(m,n,\theta_{t,\ell})\geq \Omega(m^{\frac{1}{2}+\frac{1}{\ell}} n^{\frac{1}{2}})$.
In any case, Proposition~\ref{prop:theta-lower} already shows that $\theta_{t,\ell}$ is not
$(\al,\beta)$-quasi-smooth and hence is also not $(\al,\beta)$-smooth. (As far as we know, this is
the first example of a family of bipartite graphs which are not $(\al,\beta)$-smooth.)
However, $\theta_{t,\ell}$-free graphs
have similar expansion properties as $C_{2\ell}$-free graphs (see \cite{JMY}, Lemma 4.1).
By using Lemma 4.1 in \cite{JMY} instead of Lemma~\ref{lem:jacques} in this paper, one can develop an analogous lemma as Lemma~\ref{lem:C2ell-free}.
Then using essentially the same proof as that of Theorem~\ref{thm:main2}, one can show the following.

\begin{theorem}
Let $t,\ell\geq 2$. Let  $\delta>0$ be any real. Let $k_0=3\ell(8\ell/\delta)^\ell+2\ell+2$.
For all odd integers $k\geq k_0$ and $n$ sufficiently large the following is true. If $G$ is an $n$-vertex
$\{\theta_{t,\ell},C_k\}$-free graph with minimum degree at least $\delta n^{1/\ell}$, then $G$ is bipartite.
\end{theorem}

\item Our proof method works any family $\F$ of bipartite graphs satisfying the following property.
\begin{itemize}
\item [(P1).] For any $\delta>0$, there are constants $K$ and $\mu$ such that
for every $n$-vertex $\F$-free graph $G$ with minimum degree $\delta \ex(n,\F)/n$ and for each vertex $u$ in $G$,
there are at least $\mu n$ vertices within distance $K$ from $u$.
\end{itemize}
Analogous theorems as Theorems~\ref{thm:main} and~\ref{thm:main2} hold for $\F$-free graphs. Note that $(\al, \beta)$-quasi-smooth families and theta graphs (which are not always quasi-smooth) both satisfy (P1).

\item  The proof of Theorem~\ref{thm:improve} can be generalized a bit further to yield the following.
Suppose $\F$ is a family of bipartite graphs satisfying the property (P1) and the following property (P2).
\begin{itemize}
    \item [(P2).] There exists some constants $\lambda>0$ and $2>\alpha>\beta\geq 1$ such that $z(n,\F)=\lambda n^\alpha+O(n^\beta)$.
\end{itemize}
Then Conjecture \ref{conj:ES} holds for $\F$ in the following form:
for any odd $k\geq k_0$ and sufficiently large $n$, any $n$-vertex $\F\cup \{C_k\}$-free extremal graph can be made bipartite by deleting a set of $O(n^{1+\beta-\alpha})$ vertices, which together are incident to $O(n^\beta)$ edges.

In particular, this also applies to $\F=\{C_4, C_6,...,C_{2\ell}\}$ for $\ell\in \{2,3,5\}$.

\item One could also prove our main theorems using the sparse regularity lemma (\cite{Koha}, \cite{Scott}).
However, the proofs would be more technical and would involve longer buildups. We chose to present a proof that avoids the use of sparse regularity. It seems, however, in order to make more progress on the original conjecture of Erd\H{o}s and Simonovits (Conjecture~\ref{conj:ES}), for instance to verify the conjecture for $(\al,\beta)$-quasi-smooth families, sparse regularity lemma may still be an effective tool. This is because for $(\al,\beta)$-quasi-smooth families $\F$,
like for $(\al,\beta)$-smooth families (see~\cite{AKSV}), there is a transference of density from an $\F$-free host graph to the corresponding cluster graph.

\end{enumerate}

{\noindent \bf Acknowledgement}. The authors would like to thank the referrers for many helpful comments.

\end{document}